\documentclass[11pt,reqno]{amsart}
\usepackage{amsmath}
\usepackage{cases}
\usepackage{mathrsfs}
\usepackage{bbm}
\usepackage{amssymb}
\usepackage{amscd}
\usepackage{amsfonts,latexsym,amsmath,
amsthm,amsxtra,mathdots,amssymb,latexsym,mathabx}
\usepackage[all,cmtip]{xy}
\RequirePackage{amsmath} \RequirePackage{amssymb}
\usepackage{color}
\usepackage{colordvi}
\usepackage{multicol}
\usepackage{hyperref}
\usepackage{mathtools}
\usepackage[margin=1in]{geometry}
\usepackage{xcolor}

\hypersetup{
    colorlinks,
    linkcolor={red!50!black},
    citecolor={blue!100!black},
    urlcolor={blue!100!black}
}
\usepackage{cite}

\newcommand{\bea}{\begin{eqnarray}}
\newcommand{\eea}{\end{eqnarray}}
\newcommand{\bna}{\begin{eqnarray*}}
\newcommand{\ena}{\end{eqnarray*}}

\numberwithin{equation}{section}

\theoremstyle{plain}
\newtheorem{lemma}{Lemma}[section]
\newtheorem{theorem}[lemma]{Theorem}
\newtheorem{corollary}[lemma]{Corollary}

\theoremstyle{definition}

\newtheorem{remark}{Remark}

\renewcommand{\Re}{\operatorname{Re}}
\renewcommand{\Im}{\operatorname{Im}}

\title[]
{On an unconditional $\mathrm{GL}_3$ analog of Selberg's result}

\author{Qingfeng Sun}
	\address{School of Mathematics and Statistics, Shandong University, Weihai\\Weihai, 264209, China}
	\email{qfsun@sdu.edu.cn}

\author{Hui Wang\textsuperscript{\dag}}
 \address{School of Mathematics, Shandong University, Jinan 250100, China}
 \address{Current address: Department of Mathematics,
The University of Hong Kong, Pokfulam Road, Hong Kong, China}
    \email{wh0315@mail.sdu.edu.cn}

\thanks{\dag~ Corresponding author}


\date{}

\begin{document}
\begin{abstract}
Let $F$ be a Hecke--Maass cusp form for $\mathrm{SL}_3(\mathbb{Z})$
with the Langlands parameter $\mu_{F}=\big(\mu_{F,1},\mu_{F,2},\mu_{F,3}\big)$ and the
associated $L$-function $L(s, F)$. Define $S_F(t)=\pi^{-1}\arg L(1/2+\mathrm{i}t, F)$.
When $\mu_{F}$ is in generic position, we establish an unconditional asymptotic formula
for the moments of $S_F(t)$.
Previously, such a formula was
only known to hold under the Generalized Riemann Hypothesis.
The key ingredient is a weighted zero-density
estimate in the spectral aspect for $L(s, F)$,
which has recently been proved by Sun and Wang in arXiv:2412.02416.
\end{abstract}
\maketitle

\section{Introduction}
\setcounter{equation}{0}

Let $L(s,F)$ be an $L$-function and define
\begin{equation*}
\begin{split}
S_F(t)=\frac{1}{\pi}\arg L\Big(\frac{1}{2}+\mathrm{i}t, F\Big).
\end{split}
\end{equation*}
Here,
the argument is defined by continuous variation along the horizontal
segment from $\infty+\mathrm{i}t$ to $1/2+\mathrm{i}t$, commencing at $\infty+\mathrm{i}t$ with the value $0$.
It is very interesting to study the function $S_F(t)$ in number theory.
For example,
in the case of the Riemann zeta function, the classical Riemann--von Mangoldt
formula states that
the number of zeros $\rho=\beta+\mathrm{i}\gamma$ of $\zeta(s)$ (counted with multiplicity) in the rectangle
$\{(\beta,\gamma)\mid 0<\beta<1, 0<\gamma\leq t\}$ is given by
$$
\frac{t}{2\pi}\log \frac{t}{2\pi \mathrm{e}}+\frac{7}{8}+\frac{1}{\pi}\arg \zeta(1/2+\mathrm{i}t)
+O\Big(\frac{1}{t}\Big),
$$
which implies that the function $S_{\zeta}(t)=\pi^{-1}\arg \zeta(1/2+\mathrm{i}t)$ is
closely related to the nontrivial zeros of $\zeta(s)$.
As such, the fascinating behavior of $S_{\zeta}(t)$ has been intensively studied by
many number theorists (see, e.g., \cite{Backlund,Titchmarsh1,BS,CCM} and the references therein).
On the other hand, the function $S_F(t)$ has a close connection with random matrix theory
(see \cite{BK}). In fact,
the behavior of $S_F(t)$ is analogous to the
extreme value statistics of the characteristic polynomials of random matrices.

To understand more about $S_F(t)$, an effective way is to calculate and analyze the various moments of $S_F(t)$.
In the 1940s, Selberg \cite{Selberg, Selberg1} examined the distribution of $S_{\zeta}(t)$
and demonstrated that for any $n\in \mathbb{N}$,
\begin{equation}\label{S(t) even moment}
\frac{1}{T}\int_{T}^{2T}S_{\zeta}(t)^{2n}\mathrm{d}t
=\frac{(2n)!}{n!(2\pi)^{2n}}(\log\log T)^n+O\big((\log\log T)^{n-1/2}\big),
\end{equation}
which implies that, in an average sense, $|S_{\zeta}(t)|$ has an order of magnitude $\sqrt{\log\log T}$.
Additionally, for the Dirichlet $L$-function $L(s,\chi)$ with primitive $\chi$,
Selberg \cite{Selberg2} studied the behavior of $S_{\chi}(t)=\frac{1}{\pi}\arg L\big(\frac{1}{2}+\mathrm{i}t, \chi\big)$
and obtained, for prime modulus $q$ and any fixed $t>0$,
\begin{equation}\label{S(chi) even moment}
\frac{1}{q}\,\,\sideset{}{^\star}\sum_{\chi \bmod q} S_{\chi}(t)^{2n}=\frac{(2n)!}{n!(2\pi)^{2n}}(\log\log q)^n+O\big((\log\log q)^{n-1/2}\big),
\end{equation}
where $\star$ indicates that the summation runs over the primitive characters $\chi$ modulo $q$.
It should be noted that the counterparts of \eqref{S(t) even moment} and \eqref{S(chi) even moment}
for the odd-order moments
can also be derived by making minor adjustments to Selberg's approach.

For the $\mathrm{GL}_2$ family of $L$-functions,
Hejhal and Luo \cite{HL} first considered the generalization of \eqref{S(chi) even moment}
to the case of eigenfunctions $\psi$ of the hyperbolic Laplacian $\Delta$ for the full modular group.
More precisely, let $L(s,\psi)$ be the standard Hecke--Maass $L$-function associated to $\psi$
and define
\begin{equation*}
S_{\psi}(t)=\frac{1}{\pi}\arg L\Big(\frac{1}{2}+\mathrm{i}t, \psi\Big).
\end{equation*}
They established the analogue of \eqref{S(t) even moment}
when $\psi$ is an Eisenstein series.
In addition, when $\psi$ is the Hecke--Maass cusp form $u_j$ of $\mathrm{SL}_2(\mathbb{Z})$ with the first Fourier
coefficient $\nu_j(1)$
and Laplacian eigenvalue $\lambda_j=1/4+t_j^2$ ($t_j\geq0$),
under the Generalized Riemann Hypothesis (GRH),
they showed that for any fixed $t, \beta>0$ and any large positive parameter $T$,
\begin{equation*}\label{HL result}
\begin{split}
\sum_{j=1}^\infty \mathrm{e}^{-\beta(t_j-T)^2}|\nu_j(1)|^2S_{u_j}(t)^{2n}=
\frac{2\pi^{-3/2}}{\sqrt{\beta}}\frac{(2n)!}{n!(2\pi)^{2n}}T(\log\log T)^{n}
+O_{t,n,\beta}\big(T(\log\log T)^{n-1/2}\big).
\end{split}
\end{equation*}
Recently, Sun and Wang \cite{SW2} reestablished this asymptotic
formula by means of a weighted zero-density
estimate instead of assuming the GRH.
Moreover, for the holomorphic case,
Liu and Shim \cite{LS} and Sun \emph{et al.} \cite{SW} successively obtained unconditional analogues
in the weight aspect and in the level aspect, respectively.

For the $\mathrm{GL}_3$ family of $L$-functions, Liu and Liu \cite{LL} explored the average behavior
of the analogue function $S_F(t)$, where
$F$ is a Hecke--Maass cusp form for $\mathrm{SL}_3(\mathbb{Z})$.
To state their results, we need to introduce some notations.
For a cusp form $F$ in an orthogonal basis of Hecke--Maass cusp forms for $\mathrm{SL}_3(\mathbb{Z})$,
we denote its Langlands parameter by $\mu_{F}=\big(\mu_{F,1},\mu_{F,2},\mu_{F,3}\big)$, where
$\mu_{F}$ is a point in the region
\begin{equation*}
\begin{split}
\Lambda_{1/2}'=\Bigg\{ (\mu_{F,1},\mu_{F,2},\mu_{F,3})\in\mathbb{C}^3,\,\,\,
\begin{aligned}
& \mu_{F,1}+\mu_{F,2}+\mu_{F,3}= 0, \,\,|\Re(\mu_{F,j})| \leq \frac{1}{2}, \, j = 1, 2, 3, \\
& \{-\mu_{F,1}, -\mu_{F,2}, -\mu_{F,3}\} = \{\overline{\mu_{F,1}}, \overline{\mu_{F,2}}, \overline{\mu_{F,3}}\}
\end{aligned}
\Bigg\}
\end{split}
\end{equation*}
in the Lie algebra $\mathfrak{sl}_3(\mathbb{C})$.
Denote by $\nu_{F}=\big(\nu_{F,1},\nu_{F,2},\nu_{F,3}\big)$ the spectral parameter of $F$,
where
\begin{equation}\label{eqn:n2m}
\begin{split}
\nu_{F,1}:=\frac{1}{3}&\big(\mu_{F,1}-\mu_{F,2}\big),\quad
\nu_{F,2}:=\frac{1}{3}\big(\mu_{F,2}-\mu_{F,3}\big),\\
\nu_{F,3}:=&-\nu_{F,1}-\nu_{F,2}=\frac{1}{3}\big(\mu_{F,3}-\mu_{F,1}\big).
\end{split}
\end{equation}
Then we have
\begin{equation}\label{eqn:m2n}
  \mu_{F,1}=2\nu_{F,1}+\nu_{F,2}, \quad
  \mu_{F,2}=\nu_{F,2}-\nu_{F,1}, \quad
  \mu_{F,3}=-\nu_{F,1}-2\nu_{F,2}.
\end{equation}
We will use the coordinates $\mu$ and $\nu$ simultaneously.
Note that $\big(\nu_{F,1}-1/3, \nu_{F,2}-1/3\big)$ is the type of $F$
as defined in Goldfeld's book \cite{Goldfeld}.
Assume that $F$ is normalized so that the $(1,1)$-th Fourier coefficient $A_F(1,1)=1$.
The $L$-function associated to $F$ is defined by
\[
L(s,F)=\sum_{n=1}^{\infty}\frac{A_{F}(1,n)}{n^{s}},
\quad\,\Re(s)>1,
\]
which has the Euler product
\begin{equation*}
\begin{split}
L(s,F)&=\prod_{p}\bigg(1-\frac{A_F(1,p)}{p^s}
+\frac{A_F(p,1)}{p^{2s}}-\frac{1}{p^{3s}}\bigg)^{-1}\\
&=\prod_{p}\big(1-\alpha_{F,1}(p)p^{-s}\big)^{-1}
\big(1-\alpha_{F,2}(p)p^{-s}\big)^{-1}\big(1-\alpha_{F,3}(p)p^{-s}\big)^{-1}.
\end{split}
\end{equation*}
Here we have
\begin{equation}\label{local parameter}
\begin{split}
\begin{cases}
\alpha_{F,1}(p)\alpha_{F,2}(p)\alpha_{F,3}(p)=1,\\
\alpha_{F,1}(p)+\alpha_{F,2}(p)+\alpha_{F,3}(p)=A_F(1,p),\\
\alpha_{F,1}(p)\alpha_{F,2}(p)+\alpha_{F,2}(p)\alpha_{F,3}(p)+\alpha_{F,3}(p)\alpha_{F,1}(p)=A_F(p,1),
\end{cases}
\end{split}
\end{equation}
and by Kim and Sarnak \cite{KS},
\begin{equation}\label{bound3}
\begin{split}
|\alpha_{F,j}(p)|\leq p^{\frac{5}{14}},\quad\text{and}\quad |\Re(\mu_{F,j})|\leq\frac{5}{14},\quad\text{for}\,\,j=1,2,3.
\end{split}
\end{equation}
Note that $L(s,F)$ has an analytic continuation to the whole complex plane and satisfies
the functional equation
\begin{equation}\label{eqn:FE}
\Lambda(s,F)=L(s,F)\prod_{j=1}^{3}\Gamma_{\mathbb{R}}(s-\mu_{F,j})
=L(1-s,\widetilde{F})\prod_{j=1}^{3}\Gamma_{\mathbb{R}}(1-s+\mu_{F,j})=\Lambda(1-s,\widetilde{F}),
\end{equation}
where $\Gamma_{\mathbb{R}}(s) = \pi^{-s/2}\Gamma(s/2)$ and
the dual form of $F$ is denoted by $\widetilde{F}$ with the Langlands parameter
$(-\mu_{F,1},-\mu_{F,2},-\mu_{F,3})$, the spectral parameter $\big(\nu_{F,2},\nu_{F,1},\nu_{F,3}\big)$ and the coefficients $A_{\widetilde{F}}(m,n)=\overline{A_F(m,n)}=A_F(n,m)$.

Let $\mu_0=(\mu_{0,1},\mu_{0,2},\mu_{0,3})$ and $\nu_0=(\nu_{0,1},\nu_{0,2},\nu_{0,3})$, which satisfy
the corresponding relations \eqref{eqn:n2m} and \eqref{eqn:m2n}.
We also assume that $\mu_0$ is in generic position, i.e.,
\begin{equation}\label{relation}
|\mu_{0,j}|\asymp|\nu_{0,j}|\asymp T:=\|\mu_0\|, \quad
1\leq j \leq 3,
\end{equation}
where
$$
\|\mu_0\|=\sqrt{|\mu_{0,1}|^2+|\mu_{0,2}|^2+|\mu_{0,3}|^2}.
$$
Throughout this paper, we let $M\!=\!T^{1-\eta}$ for some sufficiently small $0<\eta<1/2$.
Define a test function $h_{T,M}(\mu)$ (depending on $\mu_0$) for $\mu=(\mu_1,\mu_2,\mu_3)\in\mathbb{C}^3$ by
$$
h_{T,M}(\mu):=P(\mu)^2\bigg(\sum_{w\in\mathcal{W}}\psi\bigg(\frac{w(\mu)-\mu_0}{M}\bigg)\bigg)^2,
$$
where
$
\psi(\mu)=\exp\big(\mu_{1}^2+\mu_{2}^2+\mu_{3}^2\big)
$
and
$$
P(\mu)=\prod_{0\leq n\leq A} \prod_{j=1}^{3}
\frac{\big(\nu_{j}-\frac{1}{3}(1+2n)\big)\big(\nu_{j}+\frac{1}{3}(1+2n)\big)}{|\nu_{0,j}|^2}
$$
for some fixed large $A>0$. Here
\[
\mathcal{W} := \Big\{ I,\; w_2=\Big(\begin{smallmatrix} 1 & & \\  & & -1 \\  &1&
\end{smallmatrix}\Big),\;
w_3=\Big(\begin{smallmatrix}  &1& \\  -1&& \\  &&1   \end{smallmatrix}\Big),\;
w_4=\Big(\begin{smallmatrix}  &1& \\  &&1 \\  1&&   \end{smallmatrix}\Big),\;
w_5=\Big(\begin{smallmatrix}  &&1 \\  1&& \\  &1&   \end{smallmatrix}\Big),\;
w_6=\Big(\begin{smallmatrix}  &&1 \\  &-1& \\  1&&   \end{smallmatrix}\Big)
\Big\}
\]
is the Weyl group of $\mathrm{SL}_3(\mathbb{R})$.
The function $h_{T,M}(\mu)$ localizes at a ball of radius $M$ about $w(\mu_0)$ for each $w\in\mathcal{W}$.

Moreover, we define the spectral measure on the hyperplane $\mu_1+\mu_2+\mu_3=0$ by
\[
\mathrm{d}_{\text{spec}}\mu=\text{spec}(\mu)\mathrm{d}\mu,
\]
where
$$
\text{spec}(\mu):= \prod_{j=1}^{3}\bigg(3\nu_j\tan\bigg(\frac{3\pi}{2}\nu_j\bigg)\bigg)\quad\mbox{and}\quad
\mathrm{d}\mu=\mathrm{d}\mu_1\mathrm{d}\mu_2=\mathrm{d}\mu_2\mathrm{d}\mu_3=\mathrm{d}\mu_1\mathrm{d}\mu_3.
$$
By a trivial estimate, we have
\begin{equation}\label{property2 for h}
\mathcal{H}:=\frac{1}{192\pi^5}\int_{\Re(\mu)=0}h_{T,M}(\mu)\mathrm{d}_{\rm{spec}}\mu\asymp T^3M^2.
\end{equation}

Fix any $t>0$ and $n\in \mathbb{N}$, under the GRH, Liu and Liu \cite{LL} established that, for large positive $T$,
\begin{equation}\label{LL result}
\begin{split}
\frac{1}{\mathcal{H}}\sum_{F}\frac{h_{T,M}(\nu_{F})}{L(1,\text{Ad}^2F)}S_F(t)^{2n}=
\frac{9(2n)!}{4n!(2\pi)^{2n}}(\log\log T)^{n}
+O_{t,n}\big((\log\log T)^{n-1/2}\big),
\end{split}
\end{equation}
where $L(s,\text{Ad}^2F)$ denotes the adjoint square $L$-function of $F$.
Recently, with the automorphic Plancherel density theorem of Matz and Templier,
Chen \emph{et al.} \cite{CLW} extended their result to $\mathrm{GL}_n$.
In this paper, we are devoted to proving the asymptotic formula \eqref{LL result} without assuming the GRH.
We introduce the normalizing factor
$$
\mathcal{N}_F:=\|F\|^2\prod_{j=1}^{3}\cos\bigg(\frac{3}{2}\pi\nu_{F,j}\bigg).
$$
By Rankin--Selberg theory in combination with Stade's formula (see e.g.,
Blomer \cite[Section 4]{Blomer} and Li \cite[Theorem 2]{Li10}),
it is easy to see that
\begin{equation*}
\begin{split}
\mathcal{N}_F\asymp \mathrm{Res}_{s=1}L(s, F\times \widetilde{F})\ll \|\mu_F\|^\varepsilon,
\end{split}
\end{equation*}
where $L(s, F\times \widetilde{F})$ is the Rankin--Selberg square of
$F$ defined by
\begin{equation*}
\begin{split}
L(s,F \times \widetilde{F})&=\zeta(3s)\sum_{m,n\ge 1} \frac{|A_F(m,n)|^2}{(m^2n)^s}\\
&=\prod_p \prod_{i=1}^3\prod_{j=1}^3\bigg(1-\alpha_{F,i}(p)
\overline{\alpha_{F,j}(p)} p^{-s}\bigg)^{-1}
,\quad \text{for}\,\, \,\Re(s)>1,
\end{split}
\end{equation*}
and the implied constants depend at most on $\varepsilon$.
Our main result is the following theorem.
\begin{theorem}\label{main-theorem}
Let $t>0$ be fixed and $T$ be any large positive parameter. We have
\begin{equation}\label{main result}
\begin{split}
\frac{1}{\mathcal{H}}\sum_{F}\frac{h_{T,M}(\mu_F)}{\mathcal{N}_F} S_F(t)^n=C_n(\log\log T)^{n/2}
+O_{t,n}\big((\log\log T)^{(n-1)/2}\big),
\end{split}
\end{equation}
where
\begin{equation*}
C_n=\begin{cases}
\frac{n!}{(n/2)!(2\pi)^{n}}, \,&\textit{if }\, n\,\, \text{is even},\\
0, \,& \textit{if }\, n \,\,\text{is odd}.
\end{cases}
\end{equation*}

Moreover, define the following probability measure
$\mu_{TM}$ on $\mathbb{R}$ by
\begin{equation*}
\begin{split}
\mu_{TM}(E):=\bigg(\sum_{F}\frac{h_{T,M}(\mu_F)}{\mathcal{N}_F}
\mathbf{1}_E\bigg(\frac{S_F(t)}{\sqrt{\log\log T}}\bigg)\bigg)
\bigg/\bigg(\sum_{F}\frac{h_{T,M}(\mu_F)}{\mathcal{N}_F}\bigg),
\end{split}
\end{equation*}
where $\mathbf{1}_E$ is the characteristic function of a Borel measurable set $E$ in $\mathbb{R}$.
Then for every $a<b$, we have
\begin{equation*}
\lim_{T\rightarrow \infty} \mu_{TM}([a,b])=\sqrt{\pi}\int_a^b \exp(-\pi^2\xi^2)\mathrm{d}\xi,
\end{equation*}
uniformly with respect to $M$.
\end{theorem}
\begin{remark}
For $F$ a Hecke--Maass cusp form of $\mathrm{SL}_3(\mathbb{Z})$,
our result indicates that, as $T\rightarrow \infty$,
the ratios $S_F(t)/\sqrt{\log\log T}$ converge in distribution
to a normal distribution of mean $0$
and variance $(2\pi^2)^{-1}$.
\end{remark}

\noindent\textbf{Notation}
Throughout the paper,
$\varepsilon$ is an arbitrarily small positive
real number, which may be different at each occurrence.
As usual, $e(z) = \exp (2 \pi \mathrm{i}z) = \mathrm{e}^{2 \pi \mathrm{i}z}$ and $\delta_{m,n}$ is the Kronecker symbol that detects $m=n$.
We write $f = O(g)$ or $f \ll g$ to mean $|f| \leq cg$ for some unspecified positive constant $c$.
The symbol $y\asymp Y$ is to mean that $c_1Y\leq |y| \leq c_2Y$ for some positive constants $c_1$ and $c_2$.
\\ \

\noindent\textbf{Structure of this paper.} The rest of this paper is organized as follows.
In Section \!\ref{prelim}, we give some preliminary results needed for the proof.
In Section \!\ref{approximation of S_F(t)}, we establish some lemmas and provide a crucial approximation of $S_F(t)$.
In Section \!\ref{Proof of the main theorem}, we give the detailed proof of Theorem \ref{main-theorem}.
In Section \!\ref{sec:conclusion},
we summarize our main results, highlight the significance of our work,
and discuss prospects for future research.
\section{Preliminaries}\label{prelim}
\subsection{Hecke relations}
Firstly, let us recall the following Hecke relations (see Goldfeld \cite[pg. 168]{Goldfeld})
\begin{equation*}
\begin{split}
A_F(m_1m_2, n_1n_2)&=A_F(m_1, n_1)A_F(m_2, n_2),\quad\text{if}\,\,(m_1n_1,m_2n_2)=1,\\
A_F(n, 1)A_F(m_1, m_2)&=\sum_{d_0d_1d_2=n\atop d_1|m_1,\,d_2|m_2}A_F\bigg(\frac{m_1d_0}{d_1},\frac{m_2d_1}{d_2}\bigg),\\
A_F(1,n)A_F(m_1, m_2)&=\sum_{d_0d_1d_2=n\atop d_1|m_1,\,d_2|m_2}A_F\bigg(\frac{m_1d_2}{d_1},\frac{m_2d_0}{d_2}\bigg),\\
A_F(n_1, 1)A_F(1, n_2)&=\sum_{d|(n_1,n_2)}A_F\bigg(\frac{n_1}{d},\frac{n_2}{d}\bigg).
\end{split}
\end{equation*}
Then we arrive at the following useful result, which was presented by Liu and Liu in \cite[Proposition 5.1]{LL}.
\begin{lemma}\label{lemma 4 spectral}
For each prime $p$ and each integer $n\geq 0$, we have
\begin{equation*}\label{Hecke relation 2}
\begin{split}
A_F(1,p)^n=\sum_{k,\ell\geq 0\atop k+\ell\leq n}a_{k,\ell}(n)A_F(p^k,p^\ell)
\end{split}
\end{equation*}
for some integers $a_{k,\ell}(n)\geq 0$. Furthermore, $a_{k,\ell}(n)\neq 0$ if and only if $(k,\ell)$ belongs to
\begin{equation*}
\begin{split}
E_n:=\big\{(n_2-n_1, n_3-n_2)\,|\, n_j\in \mathbb{Z}, n_3\geq n_2\geq n_1\geq 0, n_1+n_2+n_3=n\big\}.
\end{split}
\end{equation*}
\end{lemma}
\subsection{A summation formula}
We need the following summation formula,
which is an application of the $\mathrm{GL}_3$ Kuznetsov trace formula in the version of
Buttcane \cite[Theorems 2--4]{Buttcane}.
\begin{theorem}\label{application 1}
Let $P=m_1n_1m_2n_2\neq 0$.  For any $\varepsilon>0$, we have
\begin{equation*}
\begin{split}
&\sum_{F} \frac{h_{T,M}(\mu_{F})}{\mathcal{N}_F} \overline{A_{F}(m_1, m_2)} A_{F}(n_1, n_2)\\
&\quad=\delta_{n_1, m_1} \delta_{n_2, m_2}\mathcal{H}
+O\bigg((TP)^\varepsilon\big(TP^{1/2}+T^3+TM^2P^\vartheta\big)\bigg),
\end{split}
\end{equation*}
where $\vartheta=7/64$, due to Kim and Sarnak \cite{KS}, is the best record towards the Ramanujan--Petersson conjecture for
Fourier coefficients of $\mathrm{GL}_2$ Maass cusp forms.
\end{theorem}
\begin{proof}
See Sun and Wang \cite[Theorem 7.1]{SW1}.
\end{proof}

We quote here the frequently used estimate
$$\mathcal{H}=CT^3M^2+O(T^2M^3),$$
for some constant $C>0$ (see Buttcane and Zhou \cite[pg.\,15]{BZ}). Taking $m_1=n_1=m_2=n_2=1$ in Theorem \ref{application 1} immediately gives
\begin{corollary}\label{corollary}
Let the assumptions be as in Theorem \ref{application 1}. We have
\begin{equation*}
\begin{split}
&\sum_{F} \frac{h_{T,M}(\mu_{F})}{\mathcal{N}_F}
=CT^3M^2+O(T^2M^3).
\end{split}
\end{equation*}
\end{corollary}

\subsection{The zero-density estimate}
In the context of our proof, the key point lies in the following weighted
version of zero-density estimate
in the spectral aspect for the $L$-functions associated to $\mathrm{GL}_3$ Hecke--Maass cusp forms,
which was rigorously established by Sun and Wang in \cite[Theorem 1.1]{SW1}.
\begin{lemma}\label{zero-density spectral}
For $\sigma>0$ and $H>0$, we let
\begin{equation*}
\begin{split}
N(\sigma, H; F):=\#\big\{\rho=1/2+\beta+\mathrm{i}\gamma: L(\rho, F)=0, \beta>\sigma, |\gamma|<H\big\}
\end{split}
\end{equation*}
and
\begin{equation*}\label{N(sigma,H)}
\begin{split}
N(\sigma, H):=\frac{1}{\mathcal{H}}\sum_{F}\frac{h_{T,M}(\mu_F)}{\mathcal{N}_F}N(\sigma, H; F).
\end{split}
\end{equation*}
Let $2/\log T<\sigma<1/2$. For some sufficiently small $\delta_1,\,\theta>0$, we have
\begin{equation*}
\begin{split}
N(\sigma, H)\ll HT^{-\theta\sigma}\log T,
\end{split}
\end{equation*}
uniformly in $3/\log T<H<T^{\delta_1}$.
\end{lemma}

\section{An approximation of $S_F(t)$}\label{approximation of S_F(t)}
In this section we prove several lemmas and derive an approximation of $S_F(t)$.
We denote by $\rho=\beta+\mathrm{i}\gamma$ a typical (non-trivial) zero of $L(s,F)$ inside the critical strip,
i.e., $0<\beta<1$.
For $\Re(s)>1$, we have the Dirichlet series expansion for the logarithmic derivative of $L(s,F)$,
\begin{equation}\label{Log derivative definition spectral}
-\frac{L'}{L}(s,F)=\sum_{n=1}^\infty \frac{\Lambda(n)C_F(n)}{n^s},
\end{equation}
where $\Lambda(n)$ denotes the von Mangoldt function, and
\begin{equation}\label{Log derivative coefficient spectral}
C_F(n)=\begin{cases}
\alpha_{F,1}(p)^k+\alpha_{F,2}(p)^k+\alpha_{F,3}(p)^k, \,&\text{if}\,\,n=p^k\,\,\text{for a prime}\,\,p,\\
0, \,& \text{otherwise}.
\end{cases}
\end{equation}
\begin{lemma}\label{lemma 1 spectral}
Let $x>1$. For $s\neq \rho$ and $s\neq -2m+\mu_{F,j}$ ($m\geq 0$, $j=1,2,3$), we have the following identity
\begin{equation*}
\begin{split}
\frac{L'}{L}(s,F)=&-\sum_{n\leq x^3} \frac{C_F(n)\Lambda_x(n)}{n^s}
-\frac{1}{\log^2x}\sum_\rho\frac{x^{\rho-s}(1-x^{\rho-s})^2}{(\rho-s)^3}\\
&-\frac{1}{\log^2x}\sum_{j=1}^3\sum_{m=0}^\infty\frac{x^{-2m-s+\mu_{F,j}}(1-x^{-2m-s+\mu_{F,j}})^2}{(-2m-s+\mu_{F,j})^3},
\end{split}
\end{equation*}
where
\begin{equation*}
\Lambda_x(n)=\begin{cases}
\Lambda(n), \,&\textit{if }\,\,n\leq x,\\
\Lambda(n)\frac{\log^2(x^3/n)-2\log^2(x^2/n)}{2\log^2x}, \,& \textit{if }\,\,x\leq n\leq x^2,\\
\Lambda(n)\frac{\log^2(x^3/n)}{2\log^2x}, \,& \textit{if }\,\,x^2\leq n\leq x^3,\\
0, \,& \textit{if }\,\,n\geq x^3.\\
\end{cases}
\end{equation*}
\end{lemma}
\begin{proof}
Recall the discontinuous integral
\begin{equation*}
\frac{1}{2\pi\mathrm{i}}\int_{(\alpha)}\frac{y^s}{s^3}\mathrm{d}s=
\begin{cases}
\frac{\log^2 y}{2}, \,&\textit{if }\,\, y\geq 1,\\
0, \,& \textit{if }\,\, 0<y\leq 1
\end{cases}
\end{equation*}
for $\alpha>0$.
It follows from \eqref{Log derivative definition spectral} and \eqref{Log derivative coefficient spectral} that
\begin{equation*}
-\log^2x\sum_{n=1}^\infty \frac{C_F(n)\Lambda_x(n)}{n^s}=
\frac{1}{2\pi \mathrm{i}}\int_{(\alpha)}\frac{x^v(1-x^v)^2}{v^3}\frac{L'}{L}(s+v,F)\mathrm{d}v,
\end{equation*}
where $\alpha=\max \{2, 1+\Re(s)\}$. By moving the line of integration all way to the left, we
pick up the residues at $v=0$, $v=\rho-s$ and $v=-2m-s+\mu_{F,j}$ ($m\geq 0$, $j=1,2,3$), and deduce that
\begin{equation*}
\begin{split}
&\frac{1}{2\pi \mathrm{i}}\int_{(\alpha)}\frac{x^v(1-x^v)^2}{v^3}\frac{L'}{L}(s+v,F)\mathrm{d}v\\
=&\frac{L'}{L}(s,F)\log^2x+\sum_\rho\frac{x^{\rho-s}(1-x^{\rho-s})^2}{(\rho-s)^3}
+\sum_{j=1}^3\sum_{m=0}^\infty\frac{x^{-2m-s+\mu_{F,j}}(1-x^{-2m-s+\mu_{F,j}})^2}{(-2m-s+\mu_{F,j})^3}.
\end{split}
\end{equation*}
This completes the proof of the lemma.
\end{proof}
\begin{lemma}\label{lemma 2 spectral}
For $s=\sigma+\mathrm{i}t$ and $s'=\sigma'+\mathrm{i}t'$ such that
$1/2\leq \sigma, \sigma'\leq 10$, $s\neq \rho$ and $s'\neq \rho$,
we have
\begin{equation*}
\Im\bigg(\frac{L'}{L}(s,F)-\frac{L'}{L}(s',F)\bigg)
=\Im\sum_\rho\bigg(\frac{1}{s-\rho}-\frac{1}{s'-\rho}\bigg)+O(1),
\end{equation*}
and
\begin{equation}\label{Re spectral}
\Re\frac{L'}{L}(s,F)=\sum_\rho\frac{\sigma-\beta}{(\sigma-\beta)^2+(t-\gamma)^2}+O(\log (|t|+M_F)),
\end{equation}
where
\begin{equation*}
\begin{split}
M_F:=2+\max_{1\leq j\leq3}|\mu_{F,j}|.
\end{split}
\end{equation*}
\end{lemma}
\begin{proof}
By the Hadamard factorization theorem for the entire function $\Lambda(s,F)$, we have
\begin{equation*}
\frac{L'}{L}(s,F)=b_F+\sum_\rho\bigg(\frac{1}{s-\rho}+\frac{1}{\rho}\bigg)
-\frac{1}{2}\sum_{j=1}^3\frac{\Gamma'}{\Gamma}\bigg(\frac{s-\mu_{F,j}}{2}\bigg)+\frac{3}{2}\log\pi,
\end{equation*}
for some $b_F\in \mathbb{C}$ with $\Re(b_F)=-\Re\sum_\rho\frac{1}{\rho}$ (see \cite[Proposition 5.7 (3)]{IK}).
Note that for $z\not\in -\mathbb{N}$,
\begin{equation}\label{log derivative gamma spectral}
\frac{\Gamma'}{\Gamma}(z)=-\gamma+\sum_{k\geq 1}\bigg(\frac{1}{k}-\frac{1}{k-1+z}\bigg),
\end{equation}
where $\gamma$ is the Euler constant.
Using \eqref{log derivative gamma spectral}, we get that for
$\frac{1}{14}\leq u\leq \frac{145}{28}$ and $v\ll1$,
\begin{equation*}
\Im\frac{\Gamma'}{\Gamma}(u+\mathrm{i}v)=\sum_{k\geq 1}\frac{v}{(k+u-1)^2+v^2}\ll 1.
\end{equation*}
Combining with the asymptotic formula of the logarithmic derivative of $\Gamma(z)$,
\begin{equation*}\label{derivative gamma spectral}
\frac{\Gamma'}{\Gamma}(z)=\log z-\frac{1}{2z}+O_\varepsilon\bigg(\frac{1}{|z|^2}\bigg),
\quad\text{for}\,\, |\arg z|\leq \pi-\varepsilon,
\end{equation*}
we complete the proof of the lemma.
\end{proof}
Let $x\geq 4$. We define
\begin{equation}\label{sigma_x spectral}
\sigma_x=\sigma_{x,F}=\sigma_{x,F,t}:=\frac{1}{2}+2\max_\rho\bigg\{\bigg|\beta-\frac{1}{2}\bigg|, \frac{5}{\log x}\bigg\},
\end{equation}
where $\rho=\beta+\mathrm{i} \gamma$ runs through the zeros of $L(s,F)$ for which
\begin{equation*}
|t-\gamma|\leq\frac{x^{3|\beta-1/2|}}{\log x}.
\end{equation*}

The following lemma is to display the dependence of $\sigma_x$ as needed.
\begin{lemma}\label{lemma 3 spectral}
Let $x\geq 4$. For $\sigma\geq \sigma_x$, we have
\begin{equation*}
\begin{split}
\frac{L'}{L}(\sigma+\mathrm{i}t,F)=&-\sum_{n\leq x^3}
\frac{C_F(n)\Lambda_x(n)}{n^{\sigma+\mathrm{i}t}}
+O\bigg(x^{1/4-\sigma/2}\bigg|\sum_{n\leq x^3}
\frac{C_F(n)\Lambda_x(n)}{n^{\sigma_x+\mathrm{i}t}}\bigg|\bigg)\\
&+O\bigg(x^{1/4-\sigma/2}\log (|t|+M_F)\bigg),
\end{split}
\end{equation*}
and
\begin{equation*}
\sum_\rho\frac{\sigma_x-1/2}{(\sigma_x-\beta)^2+(t-\gamma)^2}=
O\bigg(\bigg|\sum_{n\leq x^3} \frac{C_F(n)\Lambda_x(n)}{n^{\sigma_x+\mathrm{i}t}}\bigg|\bigg)
+O\big(\log (|t|+M_F)\big).
\end{equation*}
\end{lemma}

\begin{proof}
By \eqref{Re spectral}, we have
\begin{equation}\label{Re sigma_x spectral}
\Re\frac{L'}{L}(\sigma_x+\mathrm{i}t,F)=\sum_\rho\frac{\sigma_x-\beta}{(\sigma_x-\beta)^2+(t-\gamma)^2}+O(\log (|t|+M_F)).
\end{equation}
Moreover, if $\rho=\beta+\mathrm{i}\gamma$ is zero of $L(s,F)$,
then $(1-\beta)+\mathrm{i}\gamma$ is also a zero of $L(s,F)$ by the functional equation \eqref{eqn:FE}.
Thus we have
\begin{align}\label{equation 1 spectral}
&\sum_\rho\frac{\sigma_x-\beta}{(\sigma_x-\beta)^2+(t-\gamma)^2}\nonumber\\
=&\,\frac{1}{2}\bigg(\sum_\rho\frac{\sigma_x-\beta}{(\sigma_x-\beta)^2+(t-\gamma)^2}
+\sum_\rho\frac{\sigma_x-(1-\beta)}{(\sigma_x-(1-\beta))^2+(t-\gamma)^2}\bigg)\nonumber\\
=&\,\bigg(\sigma_x-\frac{1}{2}\bigg)\sum_\rho\frac{(\sigma_x-1/2)^2-(\beta-1/2)^2+(t-\gamma)^2}
{((\sigma_x-\beta)^2+(t-\gamma)^2)((\sigma_x-1+\beta)^2+(t-\gamma)^2)}.
\end{align}
If $|\beta-1/2|\leq \frac{\sigma_x-1/2}{2}$, then
\begin{equation*}
\begin{split}
(\sigma_x-1/2)^2-(\beta-1/2)^2&\geq \frac{1}{2}\big((\sigma_x-1/2)^2+(\beta-1/2)^2\big)\\
&=\frac{1}{4}\big((\sigma_x-\beta)^2+(\sigma_x-1+\beta)^2\big).
\end{split}
\end{equation*}
Thus,
\begin{equation}\label{bound 1 spectral}
(\sigma_x-1/2)^2-(\beta-1/2)^2+(t-\gamma)^2\geq\frac{1}{4}\big((\sigma_x-1+\beta)^2+(t-\gamma)^2\big).
\end{equation}
If $|\beta-1/2|> \frac{\sigma_x-1/2}{2}$, then by \eqref{sigma_x spectral} 
we have
\begin{equation*}
|t-\gamma|>\frac{x^{3|\beta-1/2|}}{\log x}>3|\beta-1/2|.
\end{equation*}
Thus,
\begin{samepage}
\begin{align}\label{bound 2 spectral}
&(\sigma_x-1/2)^2-(\beta-1/2)^2+(t-\gamma)^2\nonumber\\
&=\big((\sigma_x-1/2)^2+(\beta-1/2)^2\big)+(t-\gamma)^2-2(\beta-1/2)^2\nonumber\\
&\geq\frac{1}{2}\big((\sigma_x-\beta)^2+(\sigma_x-1+\beta)^2\big)+\frac{7}{9}(t-\gamma)^2\nonumber\\
&\geq\frac{1}{4}\big((\sigma_x-1+\beta)^2+(t-\gamma)^2\big).
\end{align}
\end{samepage}

Combining \eqref{bound 1 spectral}, \eqref{bound 2 spectral} and \eqref{equation 1 spectral}, we get
\begin{equation*}
\sum_\rho\frac{\sigma_x-\beta}{(\sigma_x-\beta)^2+(t-\gamma)^2}
\geq\frac{1}{4}(\sigma_x-1/2)\sum_\rho\frac{1}
{((\sigma_x-\beta)^2+(t-\gamma)^2)}.
\end{equation*}
Using this bound and \eqref{Re sigma_x spectral}, we obtain
\begin{equation}\label{step 1 spectral}
\sum_\rho\frac{1}{((\sigma_x-\beta)^2+(t-\gamma)^2)}
\leq \frac{4}{\sigma_x-1/2}\bigg|\frac{L'}{L}(\sigma_x+\mathrm{i}t,F)\bigg|
+O\bigg(\frac{\log (|t|+M_F)}{\sigma_x-1/2}\bigg).
\end{equation}
On the other hand, by Lemma \ref{lemma 1 spectral} and the estimate \eqref{bound3},
we have, for $\sigma\geq\sigma_{x}$,
\begin{align}\label{step 2 spectral}
\frac{L'}{L}(\sigma+\mathrm{i}t,F)=&-\sum_{n\leq x^3} \frac{C_F(n)\Lambda_x(n)}{n^{\sigma+\mathrm{i}t}}
+\frac{w_F(x,\sigma,t)}{\log^2x}\sum_\rho\frac{x^{\beta-\sigma}(1+x^{\beta-\sigma})^2}
{\big((\sigma-\beta)^2+(t-\gamma)^2\big)^{3/2}}\nonumber\\
&+O\bigg(\frac{x^{5/14-\sigma}}{\log^2x}\bigg),
\end{align}
with $|w_F(x,\sigma,t)|\leq1$.
Next we claim that
\begin{equation}\label{claim spectral}
\frac{x^{\beta-\sigma}(1+x^{\beta-\sigma})^2}
{\big((\sigma-\beta)^2+(t-\gamma)^2\big)^{3/2}}
\leq 2\log x\frac{x^{1/4-\sigma/2}}{(\beta-\sigma_x)^2+(t-\gamma)^2}.
\end{equation}
If $\beta\leq \frac{\sigma_x+1/2}{2}$, then we have
\begin{equation*}
\begin{split}
\frac{x^{\beta-\sigma}(1+x^{\beta-\sigma})^2}
{\big((\sigma-\beta)^2+(t-\gamma)^2\big)^{3/2}}&\leq
\frac{4x^{1/4-\sigma/2}}
{(\sigma_x-\beta)\big((\sigma_x-\beta)^2+(t-\gamma)^2\big)}\\
&\leq \frac{8}{\sigma_x-1/2}\frac{x^{1/4-\sigma/2}}
{\big((\sigma_x-\beta)^2+(t-\gamma)^2\big)}\\
&\leq \frac{4}{5}\log x\frac{x^{1/4-\sigma/2}}
{\big((\sigma_x-\beta)^2+(t-\gamma)^2\big)}.
\end{split}
\end{equation*}
If $\beta> \frac{\sigma_x+1/2}{2}$, then we have
\begin{equation*}
|t-\gamma|>\frac{x^{3|\beta-1/2|}}{\log x}>3|\beta-1/2|\geq 3|\beta-\sigma_x|.
\end{equation*}
Thus, $(t-\gamma)^2>\frac{8}{9}\big((\beta-\sigma_x)^2+(t-\gamma)^2\big)$.
Hence
\begin{equation*}
\begin{split}
\frac{x^{\beta-\sigma}(1+x^{\beta-\sigma})^2}
{\big((\sigma-\beta)^2+(t-\gamma)^2\big)^{3/2}}&\leq
\frac{x^{\beta-\sigma}(1+x^{\beta-1/2})^2}
{|t-\gamma|(t-\gamma)^2}
\leq\frac{\log x}{x^{3|\beta-1/2|}}\frac{9}{8}
\frac{x^{\beta-\sigma}(1+x^{\beta-1/2})^2}
{(\beta-\sigma_x)^2+(t-\gamma)^2}\\
&= \frac{9}{8}(\log x)(1+x^{-(\beta-1/2)})^2\frac{x^{1/2-\sigma}}
{(\sigma_x-\beta)^2+(t-\gamma)^2}\\
&\leq \frac{9}{8}(\log x)(1+e^{-5})^2\frac{x^{1/2-\sigma}}
{(\sigma_x-\beta)^2+(t-\gamma)^2}\\
&<2(\log x)\frac{x^{1/2-\sigma}}
{(\sigma_x-\beta)^2+(t-\gamma)^2}.
\end{split}
\end{equation*}
Using \eqref{step 1 spectral} and \eqref{claim spectral}, we have
\begin{equation*}
\begin{split}
&\sum_\rho\frac{x^{\beta-\sigma}(1+x^{\beta-\sigma})^2}
{\big((\sigma-\beta)^2+(t-\gamma)^2\big)^{3/2}}\\
\leq&\, 2\log x\sum_\rho\frac{x^{1/4-\sigma/2}}{(\beta-\sigma_x)^2+(t-\gamma)^2}\\
\leq&\,8\log x\frac{x^{1/4-\sigma/2}}{\sigma_x-1/2}\bigg|\frac{L'}{L}(\sigma_x+\mathrm{i}t,F)\bigg|
+O\bigg(\frac{(\log x)x^{1/4-\sigma/2}\log(|t|+M_F)}{\sigma_x-1/2}\bigg)\\
\leq&\,\frac{4}{5}(\log x)^2x^{1/4-\sigma/2}\bigg|\frac{L'}{L}(\sigma_x+\mathrm{i}t,F)\bigg|
+O\bigg((\log x)^2x^{1/4-\sigma/2}\log\big(|t|+M_F\big)\bigg).
\end{split}
\end{equation*}
Inserting this bound into \eqref{step 2 spectral}, we have
\begin{align}\label{step 3 spectral}
\frac{L'}{L}(\sigma+\mathrm{i}t,F)=&-\sum_{n\leq x^3} \frac{C_F(n)\Lambda_x(n)}{n^{\sigma+\mathrm{i}t}}
+\frac{4}{5}w'_F(x,\sigma,t)x^{1/4-\sigma/2}\frac{L'}{L}(\sigma_x+\mathrm{i}t,F)\nonumber\\
&+O\bigg(x^{1/4-\sigma/2}\log(|t|+M_F)\bigg),
\end{align}
with $|w'_F(x,\sigma,t)|\leq 1$.
By taking $\sigma=\sigma_x$,
\begin{equation*}
\begin{split}
&\bigg(1-\frac{4}{5}w'_F(x,\sigma_x,t)x^{1/4-\sigma_x/2}\bigg)
\frac{L'}{L}(\sigma_x+\mathrm{i}t,F)\\
=&\,O\bigg(\bigg|\sum_{n\leq x^3} \frac{C_F(n)\Lambda_x(n)}{n^{\sigma_x+\mathrm{i}t}}\bigg|\bigg)
+O\bigg(x^{1/4-\sigma_x/2}\log(|t|+M_F)\bigg).
\end{split}
\end{equation*}
Since $\Big|1-\frac{4}{5}w'_F(x,\sigma_x,t)x^{1/4-\sigma_x/2}\Big|
\geq 1-\Big|\frac{4}{5}w'_F(x,\sigma_x,t)x^{1/4-\sigma_x/2}\Big|\geq \frac{1}{5}$,
we have
\begin{equation}\label{step 4 spectral}
\frac{L'}{L}(\sigma_x+\mathrm{i}t,F)
=O\bigg(\bigg|\sum_{n\leq x^3} \frac{C_F(n)\Lambda_x(n)}{n^{\sigma_x+\mathrm{i}t}}\bigg|\bigg)
+O\bigg(x^{1/4-\sigma_x/2}\log(|t|+M_F)\bigg).
\end{equation}
Then the claimed results follow from
\eqref{step 3 spectral} and \eqref{step 4 spectral},
and from \eqref{step 1 spectral} and \eqref{step 4 spectral}, respectively.
\end{proof}
The following theorem is to provide an approximation of $S_F(t)$.
\begin{theorem}\label{theorem 1 spectral}
For $t\neq 0$ and $x\geq 4$, we have
\begin{equation*}
\begin{split}
S_F(t)=&\frac{1}{\pi}\Im\sum_{n\leq x^3}\frac{C_F(n)\Lambda_x(n)}{n^{\sigma_x+\mathrm{i}t}\log n}+
O\bigg((\sigma_x-1/2)\bigg|\sum_{n\leq x^3} \frac{C_F(n)\Lambda_x(n)}{n^{\sigma_x+\mathrm{i}t}}\bigg|\bigg)\\
&+O\big((\sigma_x-1/2)\log(|t|+M_F)\big),
\end{split}
\end{equation*}
where $\sigma_x$ is defined in \eqref{sigma_x spectral}.
\end{theorem}
\begin{proof}
By the definition of $S_F(t)$, we have
\begin{equation*}
\begin{split}
\pi S_F(t)=&-\int_{1/2}^\infty\Im \frac{L'}{L}(\sigma+\mathrm{i}t,F)\mathrm{d}\sigma\\
=&-\int_{\sigma_x}^\infty\Im \frac{L'}{L}(\sigma+\mathrm{i}t,F)\mathrm{d}\sigma
-(\sigma_x-1/2)\Im \frac{L'}{L}(\sigma_x+\mathrm{i}t,F)\\
&+\int_{1/2}^{\sigma_x}\Im\bigg( \frac{L'}{L}(\sigma_x+\mathrm{i}t,F)-\frac{L'}{L}(\sigma+\mathrm{i}t,F)\bigg)\mathrm{d}\sigma\\
:=&\, J_1+J_2+J_3.
\end{split}
\end{equation*}
By Lemma \ref{lemma 3 spectral}, we have
\begin{equation*}
\begin{split}
J_1=&-\int_{\sigma_x}^\infty\Im \frac{L'}{L}(\sigma+\mathrm{i}t,F)\mathrm{d}\sigma\\
=&\int_{\sigma_x}^\infty\Im\sum_{n\leq x^3} \frac{C_F(n)\Lambda_x(n)}{n^{\sigma+\mathrm{i}t}}\mathrm{d}\sigma
+O\bigg(\bigg|\sum_{n\leq x^3} \frac{C_F(n)\Lambda_x(n)}{n^{\sigma_x+\mathrm{i}t}}\bigg|
\int_{\sigma_x}^\infty x^{1/4-\sigma/2}\mathrm{d}\sigma\bigg)\\
&+O\bigg(\log(|t|+M_F)\int_{\sigma_x}^\infty x^{1/4-\sigma/2}\mathrm{d}\sigma\bigg)\\
=&\,\Im\sum_{n\leq x^3} \frac{C_F(n)\Lambda_x(n)}{n^{\sigma_x+\mathrm{i}t}\log n}
+O\bigg(\frac{1}{\log x}\bigg|\sum_{n\leq x^3} \frac{C_F(n)\Lambda_x(n)}{n^{\sigma_x+\mathrm{i}t}}\bigg|\bigg)
+O\bigg(\frac{\log(|t|+M_F)}{\log x}\bigg).
\end{split}
\end{equation*}
Taking $\sigma=\sigma_x$ in Lemma \ref{lemma 3 spectral}, we have
\begin{equation*}
\begin{split}
J_2\leq&\,(\sigma_x-1/2)\bigg|\frac{L'}{L}(\sigma_x+\mathrm{i}t,F)\bigg|\\
\ll&\,(\sigma_x-1/2)\bigg|\sum_{n\leq x^3} \frac{C_F(n)\Lambda_x(n)}{n^{\sigma_x+\mathrm{i}t}}\bigg|
+(\sigma_x-1/2)\log(|t|+M_F).
\end{split}
\end{equation*}
Using Lemma \ref{lemma 2 spectral}, we get
\begin{equation*}
\begin{split}
&\Im\bigg( \frac{L'}{L}(\sigma_x+\mathrm{i}t,F)-\frac{L'}{L}(\sigma+\mathrm{i}t,F)\bigg)\\
=&\sum_\rho \frac{(\sigma_x-\sigma)(\sigma+\sigma_x-2\beta)(t-\gamma)}
{\big((\sigma_x-\beta)^2+(t-\gamma)^2\big)\big((\sigma-\beta)^2+(t-\gamma)^2\big)}
+O(1).
\end{split}
\end{equation*}
Hence
\begin{equation*}
\begin{split}
|J_3|\leq&\sum_\rho\int_{1/2}^{\sigma_x}
 \frac{(\sigma_x-\sigma)|\sigma+\sigma_x-2\beta||t-\gamma|}
{\big((\sigma_x-\beta)^2+(t-\gamma)^2\big)\big((\sigma-\beta)^2+(t-\gamma)^2\big)}\mathrm{d}\sigma
+O(\sigma_x-1/2)\\
\leq&\sum_\rho\frac{\sigma_x-1/2}{(\sigma_x-\beta)^2+(t-\gamma)^2}\int_{1/2}^{\sigma_x}
 \frac{|\sigma+\sigma_x-2\beta||t-\gamma|}
{(\sigma-\beta)^2+(t-\gamma)^2}\mathrm{d}\sigma
+O(\sigma_x-1/2).
\end{split}
\end{equation*}
If $|\beta-1/2|\leq\frac{1}{2}(\sigma_x-1/2)$, then for $1/2\leq\sigma\leq\sigma_x$,
\begin{equation*}
\begin{split}
|\sigma+\sigma_x-2\beta|&=|(\sigma-1/2)+(\sigma_x-1/2)-2(\beta-1/2)|\\
&\leq |\sigma-1/2|+|\sigma_x-1/2|+2|\beta-1/2|\leq 3(\sigma_x-1/2).
\end{split}
\end{equation*}
Thus,
\begin{align}\label{bound4}
\int_{1/2}^{\sigma_x}
 \frac{|\sigma+\sigma_x-2\beta||t-\gamma|}
{(\sigma-\beta)^2+(t-\gamma)^2}\mathrm{d}\sigma
&\leq 3(\sigma_x-1/2)\int_{-\infty}^\infty\frac{|t-\gamma|}
{(\sigma-\beta)^2+(t-\gamma)^2}\mathrm{d}\sigma\nonumber\\
&\leq 10(\sigma_x-1/2).
\end{align}
If $|\beta-1/2|>\frac{1}{2}(\sigma_x-1/2)$, then
\begin{equation*}
|t-\gamma|>\frac{x^{3|\beta-1/2|}}{\log x}>3|\beta-1/2|
\end{equation*}
and for $1/2\leq\sigma\leq\sigma_x$,
\begin{equation*}
\begin{split}
|\sigma+\sigma_x-2\beta|\leq |\sigma-1/2|+|\sigma_x-1/2|+2|\beta-1/2|\leq 6|\beta-1/2|.
\end{split}
\end{equation*}
Thus,
\begin{samepage}
\begin{align}\label{bound5}
\int_{1/2}^{\sigma_x}
 \frac{|\sigma+\sigma_x-2\beta||t-\gamma|}
{(\sigma-\beta)^2+(t-\gamma)^2}\mathrm{d}\sigma
&<\int_{1/2}^{\sigma_x}\frac{|\sigma+\sigma_x-2\beta|}
{|t-\gamma|}\mathrm{d}\sigma\nonumber\\
&\leq \int_{1/2}^{\sigma_x}\frac{6|\beta-1/2|}
{3|\beta-1/2|}\mathrm{d}\sigma=2(\sigma_x-1/2).
\end{align}
\end{samepage}
By the estimates in \eqref{bound4}--\eqref{bound5} and Lemma \ref{lemma 3 spectral}, we obtain
\begin{equation*}
\begin{split}
|J_3|\leq&10(\sigma_x-1/2)\sum_\rho\frac{\sigma_x-1/2}{(\sigma_x-\beta)^2+(t-\gamma)^2}
+O(\sigma_x-1/2)\\
=&O\bigg((\sigma_x-1/2)\bigg|\sum_{n\leq x^3} \frac{C_F(n)\Lambda_x(n)}{n^{\sigma_x+\mathrm{i}t}}\bigg|\bigg)
+O\big((\sigma_x-1/2)\log(|t|+M_F)\big).
\end{split}
\end{equation*}
Thus we complete the proof of the theorem.
\end{proof}
\section{Proof of Theorem \ref{main-theorem}}\label{Proof of the main theorem}
\begin{lemma}\label{lemma 6}
Let $t>0$ be fixed and $T$ be any large positive parameter.
For \(k\in\mathbb{N}\), \(x=T^{\delta/3}\) with \(0<\delta<\frac{3\theta}{8nk+3}\), we have
\begin{equation*}
\begin{split}
\frac{1}{\mathcal{H}}\sum_{F}\frac{h_{T,M}(\mu_F)}{\mathcal{N}_F}(\sigma_x-1/2)^{4n}x^{4nk(\sigma_x-1/2)}
\ll_{t,n,\delta,k}\frac{1}{(\log T)^{4n}},
\end{split}
\end{equation*}
where $\sigma_{x,F}$ is defined in \eqref{sigma_x spectral} and $\theta$ is as in Lemma \ref{zero-density spectral}.
\end{lemma}
\begin{proof}
For $k\in\mathbb{N}$, by the definition of $\sigma_{x,F}$, we have
\begin{equation*}
\begin{split}
&(\sigma_{x,F}-1/2)^{4n}x^{4nk(\sigma_{x,F}-1/2)}\\
\leq&\bigg(\frac{10}{\log x}\bigg)^{4n}x^{40nk/\log x}
+2^{4n+1}\sum_{\beta>\frac{1}{2}+\frac{5}{\log x}\atop |t-\gamma|\leq\frac{x^{3|\beta-1/2|}}{\log x}}
(\beta-1/2)^{4n}x^{8nk(\beta-1/2)}.
\end{split}
\end{equation*}
On the other hand,
\begin{equation*}
\begin{split}
&\sum_{\beta>\frac{1}{2}+\frac{5}{\log x}\atop |t-\gamma|\leq\frac{x^{3|\beta-1/2|}}{\log x}}
(\beta-1/2)^{4n}x^{8nk(\beta-1/2)}\\
\leq&\sum_{i=5}^{\frac{1}{2}\lfloor\log x\rfloor}\bigg(\frac{i+1}{\log x}\bigg)^{4n}
x^{8nk\frac{i+1}{\log x}}\sum_{\frac{1}{2}+\frac{i}{\log x}<\beta\leq\frac{1}{2}+\frac{i+1}{\log x}
\atop |t-\gamma|\leq\frac{x^{3|\beta-1/2|}}{\log x}}1\\
\leq&\frac{1}{(\log x)^{4n}}\sum_{i=5}^{\frac{1}{2}\lfloor\log x\rfloor}\big(i+1\big)^{4n}
\mathrm{e}^{8nk(i+1)}N\bigg(\frac{i}{\log x}, |t|+\frac{\mathrm{e}^{3(i+1)}}{\log x};F\bigg).
\end{split}
\end{equation*}
By Lemma \ref{zero-density spectral}, we have
\begin{equation*}
\begin{split}
&\frac{1}{\mathcal{H}}\sum_{F}\frac{h_{T,M}(\mu_F)}{\mathcal{N}_F}
\sum_{\beta>\frac{1}{2}+\frac{5}{\log x}\atop |t-\gamma|\leq\frac{x^{3|\beta-1/2|}}{\log x}}
(\beta-1/2)^{4n}x^{8nk(\beta-1/2)}\\
\leq&\,\frac{1}{(\log x)^{4n}}\sum_{i=5}^{\frac{1}{2}\lfloor\log x\rfloor}\big(i+1\big)^{4n}
\mathrm{e}^{8nk(i+1)}\frac{1}{\mathcal{H}}\sum_{F}\frac{h_{T,M}(\mu_F)}{\mathcal{N}_F}
N\bigg(\frac{i}{\log x}, |t|+\frac{\mathrm{e}^{3(i+1)}}{\log x};F\bigg)\\
\ll&\,\frac{1}{(\log x)^{4n}}\sum_{i=5}^{\infty}\big(i+1\big)^{4n}
 \mathrm{e}^{8nk(i+1)}\bigg(|t|+\frac{\mathrm{e}^{3(i+1)}}{\log x}\bigg) T^{-\theta\frac{i}{\log x}}\log T\\
\ll&_{t,n,\delta,k}\,\frac{\log T}{(\log x)^{4n+1}}\sum_{i=5}^{\infty}\big(i+1\big)^{4n}
\mathrm{e}^{(8nk+3-\frac{3\theta}{\delta})i}\\
\ll&_{t,n,\delta,k}\,\frac{1}{(\log T)^{4n}}
\end{split}
\end{equation*}
provided that $0<\delta<\frac{3\theta}{8nk+3}$.
In addition, by Corollary \ref{corollary} and \eqref{property2 for h},
\begin{equation*}
\begin{split}
\frac{1}{\mathcal{H}}\sum_{F}\frac{h_{T,M}(\mu_F)}{\mathcal{N}_F}
\bigg(\frac{10}{\log x}\bigg)^{4n}x^{40nk/\log x}
\ll_{n,k}\frac{1}{(\log x)^{4n}}\ll_{n,k,\delta}\frac{1}{(\log T)^{4n}}.
\end{split}
\end{equation*}
Thus we complete the proof of the lemma.
\end{proof}

For a positive parameter $x$ (to be chosen later), let
\begin{equation}\label{MFRF}
\mathcal{M}_F(t):=\frac{1}{\pi}\Im\sum_{p\leq x^3}\frac{C_F(p)}{p^{1/2+\mathrm{i}t}} \qquad
\text{and}\qquad
\mathcal{R}_F(t):=S_F(t)-\mathcal{M}_F(t),
\end{equation}
where we recall \eqref{Log derivative coefficient spectral} that
$C_F(p)=\alpha_{F,1}(p)+\alpha_{F,2}(p)+\alpha_{F,3}(p)=A_F(1,p)$.
\begin{lemma}\label{lemma 5 spectral} We have
\begin{equation*}
\begin{split}
\mathcal{R}_F(t)=&O\bigg(\bigg|\Im\sum_{p\leq x^3}
\frac{A_F(1,p)(\Lambda_x(p)-\Lambda(p))}{p^{1/2+\mathrm{i}t}\log p}\bigg|\bigg)
+O\bigg(\sum_{k=2}^7\bigg|\Im\sum_{p^k\leq x^3}\frac{C_F(p^k)\Lambda_x(p^k)}{p^{k(1/2+\mathrm{i}t)}\log p}\bigg|\bigg)\\
&+O\bigg((\sigma_x-1/2)\sum_{k=1}^7x^{k(\sigma_x-1/2)}
\int_{1/2}^\infty x^{k(1/2-\sigma)}\bigg|\sum_{p^k\leq x^3}
\frac{C_F(p^k)\Lambda_x(p^k)\log(xp)}{p^{k(\sigma+\mathrm{i}t)}}\bigg|\mathrm{d}\sigma\bigg)\\
&+O\big((\sigma_x-1/2)\log(|t|+M_F)\big)+O\big(1\big).
\end{split}
\end{equation*}
\end{lemma}
\begin{proof}
By Theorem \ref{theorem 1 spectral},
\begin{equation*}
\begin{split}
\mathcal{R}_F(t)=&\frac{1}{\pi}\Im\sum_{p\leq x^3}
\frac{C_F(p)(\Lambda_x(p)p^{1/2-\sigma_x}-\Lambda(p))}{p^{1/2+\mathrm{i}t}\log p}
+\frac{1}{\pi}\Im\sum_{k=2}^\infty\sum_{p^k\leq x^3}\frac{C_F(p^k)\Lambda_x(p^k)}{p^{k(\sigma_x+\mathrm{i}t)}\log p^k}\\
&+O\bigg((\sigma_x-1/2)\bigg|\sum_{k=1}^\infty\sum_{p^k\leq x^3} \frac{C_F(p^k)
\Lambda_x(p^k)}{p^{k(\sigma_x+\mathrm{i}t)}}\bigg|\bigg)
+O\big((\sigma_x-1/2)\log(|t|+M_F)\big).
\end{split}
\end{equation*}
Using the bound \eqref{bound3}, we deduce that
\begin{equation*}
\begin{split}
\sum_{k=8}^\infty\sum_{p^k\leq x^3}\frac{C_F(p^k)\Lambda_x(p^k)}{p^{k(\sigma_x+\mathrm{i}t)}}=O(1),\quad\text{and}\quad
\sum_{k=8}^\infty\sum_{p^k\leq x^3}\frac{C_F(p^k)\Lambda_x(p^k)}{p^{k(\sigma_x+\mathrm{i}t)}\log p}=O(1).
\end{split}
\end{equation*}
Thus,
\begin{equation*}
\begin{split}
\mathcal{R}_F(t)=&\,\frac{1}{\pi}\Im\sum_{p\leq x^3}
\frac{A_F(1,p)(\Lambda_x(p)-\Lambda(p))}{p^{1/2+\mathrm{i}t}\log p}
+\frac{1}{\pi}\Im\sum_{k=2}^7\sum_{p^k\leq x^3}\frac{C_F(p^k)\Lambda_x(p^k)}{p^{k(1/2+\mathrm{i}t)}\log p^k}\\
&+\frac{1}{\pi}\Im\sum_{k=1}^7\sum_{p^k\leq x^3}
\frac{C_F(p^k)\Lambda_x(p^k)(p^{k(1/2-\sigma_x)}-1)}{p^{k(1/2+\mathrm{i}t)}\log p^k}\\
&+O\bigg((\sigma_x-1/2)\bigg|\sum_{k=1}^7\sum_{p^k\leq x^3} \frac{C_F(p^k)
\Lambda_x(p^k)}{p^{k(\sigma_x+\mathrm{i}t)}}\bigg|\bigg)\\&+O\big((\sigma_x-1/2)\log(|t|+M_F)\big)+O(1).
\end{split}
\end{equation*}
Note that for $1/2\leq a\leq \sigma_x$ and $k\geq1$,
\begin{equation*}
\begin{split}
\bigg|\sum_{p^k\leq x^3}
\frac{C_F(p^k)\Lambda_x(p^k)}{p^{k(a+\mathrm{i}t)}}\bigg|
&=kx^{k(a-1/2)}\bigg|\int_a^\infty x^{k(1/2-\sigma)}\sum_{p^k\leq x^3}
\frac{C_F(p^k)\Lambda_x(p^k)\log(xp)}{p^{k(\sigma+\mathrm{i}t)}}\mathrm{d}\sigma\bigg|\\
&\leq kx^{k(\sigma_x-1/2)}\int_{1/2}^\infty x^{k(1/2-\sigma)}\bigg|\sum_{p^k\leq x^3}
\frac{C_F(p^k)\Lambda_x(p^k)\log(xp)}{p^{k(\sigma+\mathrm{i}t)}}\bigg|\mathrm{d}\sigma.
\end{split}
\end{equation*}
Thus,
\begin{equation*}
\begin{split}
&\bigg|\sum_{p^k\leq x^3}
\frac{C_F(p^k)\Lambda_x(p^k)(p^{k(1/2-\sigma_x)}-1)}{p^{k(1/2+\mathrm{i}t)}\log p^k}\bigg|
=\bigg|\int_{1/2}^{\sigma_x} \sum_{p^k\leq x^3}
\frac{C_F(p^k)\Lambda_x(p^k)}{p^{k(a+\mathrm{i}t)}}\mathrm{d}a\bigg|\\
&\leq k(\sigma_x-1/2)x^{k(\sigma_x-1/2)}\int_{1/2}^\infty x^{k(1/2-\sigma)}\bigg|\sum_{p^k\leq x^3}
\frac{C_F(p^k)\Lambda_x(p^k)\log(xp)}{p^{k(\sigma+\mathrm{i}t)}}\bigg|\mathrm{d}\sigma,
\end{split}
\end{equation*}
and
\begin{equation*}
\begin{split}
\bigg|\sum_{p^k\leq x^3} \frac{C_F(p^k)
\Lambda_x(p^k)}{p^{k(\sigma_x+\mathrm{i}t)}}\bigg|\leq kx^{k(\sigma_x-1/2)}
\int_{1/2}^\infty x^{k(1/2-\sigma)}\bigg|\sum_{p^k\leq x^3}
\frac{C_F(p^k)\Lambda_x(p^k)\log(xp)}{p^{k(\sigma+\mathrm{i}t)}}\bigg|\mathrm{d}\sigma.
\end{split}
\end{equation*}
Thus we complete the proof of the lemma.
\end{proof}
\begin{lemma}\label{lemma 6 spectral}
Let $t>0$ be fixed and $T$ be any large positive parameter.
For $x=T^{\delta/3}$ with sufficiently small $\delta>0$, we have
\begin{equation}\label{M_F(t) moment spectral}
\begin{split}
\frac{1}{\mathcal{H}}\sum_{F}\frac{h_{T,M}(\mu_F)}{\mathcal{N}_F}\mathcal{M}_F(t)^n
=C_n(\log\log T)^{\frac{n}{2}}
+O_{t,n}\big((\log\log T)^{\frac{n-1}{2}}\big)
\end{split}
\end{equation}
and
\begin{equation}\label{R_F(t) moment spectral}
\begin{split}
\frac{1}{\mathcal{H}}\sum_{F}\frac{h_{T,M}(\mu_F)}{\mathcal{N}_F} |\mathcal{R}_F(t)|^{2n}
=O_{t,n}(1),
\end{split}
\end{equation}
where the $O$-term depends only on $t$ and $n$, $C_n$ is defined by
\begin{equation*}
C_n=\begin{cases}
\frac{n!}{(n/2)!(2\pi)^{n}}, \,&\textit{if }\, n\,\, \text{is even},\\
0, \,& \textit{if }\, n \,\,\text{is odd}.
\end{cases}
\end{equation*}
\end{lemma}
\begin{proof}
Recall that
\begin{equation*}
\mathcal{M}_F(t)=\frac{1}{\pi}\Im\sum_{p\leq x^3}\frac{C_F(p)}{p^{1/2+\mathrm{i}t}}=
\frac{-\mathrm{i}}{2\pi}\bigg(\sum_{p\leq x^3}\frac{A_F(1,p)}{p^{1/2+\mathrm{i}t}}
-\sum_{p\leq x^3}\frac{\overline{A_F(1,p)}}{p^{1/2-\mathrm{i}t}}\bigg).
\end{equation*}
Set $x=T^{\delta/3}$ for a suitably small $\delta>0$. Thus we have that the general term of
\begin{equation*}
\mathcal{M}_F(t)^n=\frac{(-\mathrm{i})^n}{(2\pi)^n}\bigg(\sum_{p\leq T^{\delta}}\frac{A_F(1,p)}{p^{1/2+\mathrm{i}t}}
-\sum_{p\leq T^{\delta}}\frac{\overline{A_F(1,p)}}{p^{1/2-\mathrm{i}t}}\bigg)^n
\end{equation*}
has the form
\begin{equation}\label{general term spectral}
\frac{A_F(1,p_1)^{m(p_1)}\overline{A_F(1,p_1)^{n(p_1)}}}{p_1^{m(p_1)(1/2+\mathrm{i}t)}(-1)^{n(p_1)}p_1^{n(p_1)(1/2-\mathrm{i}t)}}
\times\cdots\times
\frac{A_F(1,p_r)^{m(p_r)}\overline{A_F(1,p_r)^{n(p_r)}}}{p_r^{m(p_r)(1/2+\mathrm{i}t)}(-1)^{n(p_r)}p_r^{n(p_r)(1/2-\mathrm{i}t)}},
\end{equation}
where $p_1<p_2<\cdots<p_r\leq T^\delta$, $m(p_j)+n(p_j)\geq 1$, and $\sum_{j=1}^r(m(p_j)+n(p_j))=n$.

Now we discuss the contribution from the general term \eqref{general term spectral} in the following three cases.

Case\,(\uppercase\expandafter{\romannumeral 1}):
\emph{In the general term \eqref{general term spectral}, $m(p_{j_0})\not\equiv n(p_{j_0})(\bmod\,3)$ for some $j_0$.}
By Lemma \ref{lemma 4 spectral}, we have
\begin{equation*}
\begin{split}
A_F(1,p_{j_0})^{m(p_{j_0})}=\sum_{k,\ell\geq 0\atop k+\ell\leq m(p_{j_0})}a_{k,\ell}(m(p_{j_0}))A_F(p_{j_0}^k,p_{j_0}^\ell)
\end{split}
\end{equation*}
and
\begin{equation*}
\begin{split}
\overline{A_F(1,p_{j_0})^{n(p_{j_0})}}=\sum_{k,\ell\geq 0\atop k+\ell\leq n(p_{j_0})}a_{k,\ell}(n(p_{j_0}))\overline{A_F(p_{j_0}^k,p_{j_0}^\ell)}.
\end{split}
\end{equation*}
Then we can observe that there does not exist $(k,\ell)$ such that $a_{k,\ell}(m(p_{j_0}))\neq 0$
and $a_{k,\ell}(n(p_{j_0}))\neq 0$. Otherwise, by Lemma \ref{lemma 4 spectral},
there would exist $(k,\ell)\in E_{m(p_{j_0})}\cap E_{n(p_{j_0})}$ and integers $m_h,\,n_h\geq 0$, $h=1,2,3$, such that
\begin{equation*}
\begin{split}
m(p_{j_0})&=m_1+m_2+m_3,\\
n(p_{j_0})&=n_1+n_2+n_3,\\
k&=-m_1+m_2=-n_1+n_2,\\
\ell&=-m_2+m_3=-n_2+n_3.
\end{split}
\end{equation*}
There equations would further lead to
\begin{equation*}
\begin{split}
m(p_{j_0})-n(p_{j_0})&=(m_1-n_1)+(m_2-n_2)+(m_3-n_3)\\
&=3(m_2-n_2)\equiv 0(\bmod \,3),
\end{split}
\end{equation*}
which contradicts the assumption for Case\,(\uppercase\expandafter{\romannumeral 1}).

With the above observation and the Hecke relations, the numerator of this general term \eqref{general term spectral}
can be written as a sum whose terms are all of the form $A_F(m_1,n_1)\overline{A_F(m_2,n_2)}$ where $(m_1,n_1)\neq(m_2,n_2)$.
Since for these terms $\delta_{m_1, m_2} \delta_{n_1, n_2}=0$,
we see from Theorem \ref{application 1} and the estimate \eqref{property2 for h} that the contribution from the general term
\eqref{general term spectral} to the moment \eqref{M_F(t) moment spectral} is
\begin{equation*}
\begin{split}
&\ll_n\frac{1}{\mathcal{H}}\sum_{p_1<p_2<\cdots<p_r\leq T^\delta}p_1^{-\frac{1}{2}(m(p_1)+n(p_1))}\cdots
p_r^{-\frac{1}{2}(m(p_r)+n(p_r))}\\
&\qquad\qquad\qquad\qquad\qquad
\times\bigg((T^{1+2n\delta})^\varepsilon\big(T^{1+n\delta}+T^3+T^{1+2n\delta\vartheta}M^2\big)\bigg)\\
&\ll_n\frac{1}{\mathcal{H}}\bigg(\sum_{p\leq T^\delta}p^{-\frac{1}{2}}\bigg)^r
\big(T^{1+\varepsilon+n\delta(1+2\varepsilon)}+T^{3+\varepsilon+2n\delta\varepsilon}
+T^{1+\varepsilon+2n\delta(\vartheta+\varepsilon)}M^2\big)\\
&\ll_nT^{-2+\varepsilon+n\delta(3/2+2\varepsilon)}M^{-2}+T^{\varepsilon+n\delta(1/2+2\varepsilon)}M^{-2}
+T^{-2+\varepsilon+n\delta(1/2+2\vartheta+2\varepsilon)},
\end{split}
\end{equation*}
which is negligible upon taking $0<\delta<\frac{1}{n}$.

Case\,(\uppercase\expandafter{\romannumeral 2}):
\emph{In the general term \eqref{general term spectral}, $m(p_{j})\equiv n(p_{j})(\bmod\,3)$ for all $j$,
and $m(p_{j_0})\geq2$ or $n(p_{j_0})\geq2$ for some $j_0$.}
In this case, $m(p_{j})+n(p_{j})\geq 2$ for all $j$ and $m(p_{j_0})+n(p_{j_0})\geq 3$ for some $j_0$,
so $n>2r$ and $r<\lfloor\frac{n-1}{2}\rfloor$.
By applying Lemma \ref{lemma 4 spectral} to the numerator of this general term \eqref{general term spectral}
and multiplying out by Hecke relations,
one sees that it is possible to have resulting terms $A_F(m_1,n_1)\overline{A_F(m_2,n_2)}$ where $(m_1,n_1)=(m_2,n_2)$.
Using Theorem \ref{application 1}, the contribution to the moment \eqref{M_F(t) moment spectral} such terms is bounded by
the diagonal term of size $\asymp T^3M^2$. The contribution from the off-diagonal terms is clearly negligible compared
to $T^3M^2$, given that $0<\delta<\frac{1}{n}$.
Using the classical result
\begin{equation}\label{classic bound 2}
\sum_{p\leq x}\frac{1}{p}=\log\log x+O(1),
\end{equation}
we have the contribution from this general term to the moment \eqref{M_F(t) moment spectral} is at most
\begin{equation*}
\begin{split}
&\ll_n\sum_{p_1<p_2<\cdots<p_r\leq T^\delta}p_1^{-\frac{1}{2}(m(p_1)+n(p_1))}\cdots
p_r^{-\frac{1}{2}(m(p_r)+n(p_r))}\\
&\ll_n\bigg(\sum_{p\leq T^\delta}p^{-1}\bigg)^r
\ll_n (\log\log T)^{\lfloor\frac{n-1}{2}\rfloor}.
\end{split}
\end{equation*}

It remains to discuss the following: In the general term \eqref{general term spectral}, $m(p_{j})=n(p_{j})=1$ for all $j$.
Clearly $n$ must be even, say $n=2m$ and so $r=m$.

Case\,(\uppercase\expandafter{\romannumeral 3}):
\emph{$m(p_{j})=n(p_{j})=1$ for all $j$.}
The contribution from these terms to \eqref{M_F(t) moment spectral} is
\begin{align}\label{main contribution}
&\frac{1}{\mathcal{H}}\sum_{F}\frac{h_{T,M}(\mu_F)}{\mathcal{N}_F}
\binom{2m}{m}(m!)(m!)\frac{(-1)^m(-\mathrm{i})^{2m}}{(2\pi)^{2m}}\nonumber\\
&\qquad\sum_{p_1<p_2<\cdots<p_m\leq T^\delta}
\frac{A_F(1,p_1)\overline{A_F(1,p_1)}\cdots A_F(1,p_m)\overline{A_F(1,p_m)}}{p_1\cdots p_m}.
\end{align}
By the Hecke relations, we have
\begin{equation*}
\begin{split}
&A_F(1,p_1)\overline{A_F(1,p_1)}\cdots A_F(1,p_m)\overline{A_F(1,p_m)}\\
&\quad=A_F(1,p_1\cdots p_m)\overline{A_F(1,p_1\cdots p_m)}.
\end{split}
\end{equation*}
Applying Theorem \ref{application 1} and the estimate \eqref{classic bound 2}, we have \eqref{main contribution} equals
\begin{equation*}
\begin{split}
&\frac{1}{\mathcal{H}}\frac{(2m)!}{(2\pi)^{2m}}\sum_{p_1<p_2<\cdots<p_m\leq T^\delta}\frac{1}{p_1\cdots p_m}
\cdot\bigg(\mathcal{H}+O\bigg(T^{1+\varepsilon}(p_1\cdots p_m)+T^{3+\varepsilon}+T^{1+\varepsilon}M^2(p_1\cdots p_m)^{2\vartheta}\bigg)\bigg)\\
=&\frac{(2m)!}{m!(2\pi)^{2m}}\sum_{p_1,p_2,\ldots,p_m\leq T^\delta\atop p_j \,\text{distinct}}\frac{1}{p_1\cdots p_m}
+O\bigg(T^{-2+\varepsilon}M^{-2}
\sum_{p_1,p_2,\ldots,p_m\leq T^\delta\atop p_j \,\text{distinct}}1\bigg)\\
&+O\bigg(T^{\varepsilon}M^{-2}\sum_{p_1,p_2,\ldots,p_m\leq T^\delta\atop p_j \,\text{distinct}}\frac{1}{p_1\cdots p_m}\bigg)
+O\bigg(T^{-2+\varepsilon}
\sum_{p_1,p_2,\ldots,p_m\leq T^\delta\atop p_j \,
\text{distinct}}\frac{1}{(p_1\cdots p_m)^{1-2\vartheta}}\bigg)\\
=&\frac{(2m)!}{m!(2\pi)^{2m}}(\log\log T+O(1))^m
+O\big(T^{m\delta-2+\varepsilon}M^{-2}+T^\varepsilon M^{-2}(\log\log T)^m+T^{2m\delta\vartheta-2+\varepsilon}(\log\log T)^m\big)\\
=&\frac{(2m)!}{m!(2\pi)^{2m}}(\log\log T)^m
+O_m\big((\log\log T)^{m-1}\big).
\end{split}
\end{equation*}
Now the asymptotic formula \eqref{M_F(t) moment spectral} follows from
Cases (\uppercase\expandafter{\romannumeral 1})--(\uppercase\expandafter{\romannumeral 3}).

Next we are ready to prove \eqref{R_F(t) moment spectral}.
Recall that $\sigma_x=\sigma_{x,F}$ depending on $F$.
By Lemma \ref{lemma 5 spectral},
\begin{samepage}
\begin{align}\label{R}
&\frac{1}{\mathcal{H}}\sum_{F}\frac{h_{T,M}(\mu_F)}{\mathcal{N}_F} |\mathcal{R}_F(t)|^{2n}\nonumber\\
\ll&\frac{1}{\mathcal{H}}\sum_{F}\frac{h_{T,M}(\mu_F)}{\mathcal{N}_F}\bigg|\Im\sum_{p\leq x^3}
\frac{A_F(1,p)(\Lambda_x(p)-\Lambda(p))}{p^{1/2+\mathrm{i}t}\log p}\bigg|^{2n}
+\sum_{k=2}^7\frac{1}{\mathcal{H}}\sum_{F}\frac{h_{T,M}(\mu_F)}{\mathcal{N}_F}
\bigg|\Im\sum_{p^k\leq x^3}\frac{C_F(p^k)\Lambda_x(p^k)}{p^{k(1/2+\mathrm{i}t)}\log p}\bigg|^{2n}\nonumber\\
+&\sum_{k=1}^7\frac{1}{\mathcal{H}}\sum_{F}\frac{h_{T,M}(\mu_F)}{\mathcal{N}_F}
(\sigma_x-1/2)^{2n}x^{2nk(\sigma_x-1/2)}\bigg(\int_{1/2}^\infty x^{k(1/2-\sigma)}\bigg|\sum_{p^k\leq x^3}
\frac{C_F(p^k)\Lambda_x(p^k)\log(xp)}{p^{k(\sigma+\mathrm{i}t)}}\bigg|\mathrm{d}\sigma\bigg)^{2n}\nonumber\\
+&\frac{1}{\mathcal{H}}\sum_{F}\frac{h_{T,M}(\mu_F)}{\mathcal{N}_F}(\sigma_x-1/2)^{2n}\big(\log(|t|+M_F)\big)^{2n}
+\frac{1}{\mathcal{H}}\sum_{F}\frac{h_{T,M}(\mu_F)}{\mathcal{N}_F}.
\end{align}
\end{samepage}
For the first two terms in \eqref{R}, we have
\begin{equation*}
\begin{split}
&\bigg|\Im\sum_{p\leq x^3}
\frac{A_F(1,p)(\Lambda_x(p)-\Lambda(p))}{p^{1/2+\mathrm{i}t}\log p}\bigg|^{2n}\\
&\qquad\ll_n\bigg(\sum_{p\leq x^3}
\frac{A_F(1,p)(\Lambda_x(p)-\Lambda(p))}{p^{1/2+\mathrm{i}t}\log p}-\sum_{p\leq x^3}
\frac{\overline{A_F(1,p)}(\Lambda_x(p)-\Lambda(p))}{p^{1/2-\mathrm{i}t}\log p}\bigg)^{2n},
\end{split}
\end{equation*}
and for $k\geq 2$,
\begin{equation}\label{error term 2}
\begin{split}
\bigg|\Im\sum_{p^k\leq x^3}\frac{C_F(p^k)\Lambda_x(p^k)}{p^{k(1/2+\mathrm{i}t)}\log p}\bigg|^{2n}
\ll_n \bigg(\sum_{p^k\leq x^3}\frac{C_F(p^k)\Lambda_x(p^k)}{p^{k(1/2+\mathrm{i}t)}\log p}
-\sum_{p^k\leq x^3}\frac{\overline{C_F(p^k)}\Lambda_x(p^k)}{p^{k(1/2-\mathrm{i}t)}\log p}\bigg)^{2n}.
\end{split}
\end{equation}
For \eqref{error term 2}, we note that
\begin{equation*}
\begin{split}
C_F(p^k)=\alpha_{F,1}(p)^k+\alpha_{F,2}(p)^k+\alpha_{F,3}(p)^k
\end{split}
\end{equation*}
is a symmetric polynomial in $\alpha_{F,j}(p)$, since \eqref{local parameter} says that $1$, $A_F(1,p)$, $A_F(p,1)=\overline{A_F(1,p)}$
are the elementary symmetric polynomials in $\alpha_{F,j}(p)$, the fundamental theorem of symmetric polynomials
(see \cite[pg.\,21, Remark]{Macdonald}) implies
that the symmetric polynomial $C_F(p^k)=\alpha_{F,1}(p)^k+\alpha_{F,2}(p)^k+\alpha_{F,3}(p)^k$ is
a polynomial in $1$, $A_F(1,p)$ and $\overline{A_F(1,p)}$. Then we can use the fact $|\Lambda_x(p)-\Lambda(p)|=O\big(\frac{\log^3p}{\log^2x}\big)$
and the similar argument as Case (\uppercase\expandafter{\romannumeral 1})--(\uppercase\expandafter{\romannumeral 3})
to see that the contributions from first two terms in \eqref{R}
to \eqref{R_F(t) moment spectral} are of $O_{t,n}(1)$.
The contributions from last two terms in \eqref{R} are of $O_{t,n}(1)$ by Cauchy--Schwarz inequality,
the assumption \eqref{relation}, Corollary \ref{corollary} and Lemma \ref{lemma 6}.
For the third term in \eqref{R},
it follows from Cauchy--Schwarz inequality that
\begin{align}\label{after Cauchy}
&\frac{1}{\mathcal{H}}\sum_{F}\frac{h_{T,M}(\mu_F)}{\mathcal{N}_F}
(\sigma_x-1/2)^{2n}x^{2nk(\sigma_x-1/2)}\bigg(\int_{1/2}^\infty x^{k(1/2-\sigma)}\bigg|\sum_{p^k\leq x^3}
\frac{C_F(p^k)\Lambda_x(p^k)\log(xp)}{p^{k(\sigma+\mathrm{i}t)}}\bigg|\mathrm{d}\sigma\bigg)^{2n}\nonumber\\
&\leq \bigg(\frac{1}{\mathcal{H}}\sum_{F}\frac{h_{T,M}(\mu_F)}{\mathcal{N}_F}(\sigma_x-1/2)^{4n}
x^{4nk(\sigma_x-1/2)}\bigg)^{1/2}\nonumber\\
&\times\bigg(\frac{1}{\mathcal{H}}\sum_{F}\frac{h_{T,M}(\mu_F)}{\mathcal{N}_F}\bigg(\int_{1/2}^\infty
x^{k(1/2-\sigma)}\bigg|\sum_{p^k\leq x^3}
\frac{C_F(p^k)\Lambda_x(p^k)\log(xp)}{p^{k(\sigma+\mathrm{i}t)}}\bigg|\mathrm{d}\sigma\bigg)^{4n}\bigg)^{1/2},
\end{align}
where $1\leq k\leq 7$. By H\"{o}lder's inequality with the exponents $4n/(4n-1)$ and $4n$,
\begin{equation*}
\begin{split}
&\bigg(\int_{1/2}^\infty x^{k(1/2-\sigma)}\bigg|\sum_{p^k\leq x^3}
\frac{C_F(p^k)\Lambda_x(p^k)\log(xp)}{p^{k(\sigma+\mathrm{i}t)}}\bigg|\mathrm{d}\sigma\bigg)^{4n}\\
\leq&\,(\int_{1/2}^\infty
x^{k(1/2-\sigma)}\mathrm{d}\sigma\bigg)^{4n-1}\bigg(\int_{1/2}^\infty x^{k(1/2-\sigma)}\bigg|\sum_{p^k\leq x^3}
\frac{C_F(p^k)\Lambda_x(p^k)\log(xp)}{p^{k(\sigma+\mathrm{i}t)}}\bigg|^{4n}\mathrm{d}\sigma\bigg)\\
=&\,\frac{1}{(k\log x)^{4n-1}}\int_{1/2}^\infty x^{k(1/2-\sigma)}\bigg|\sum_{p^k\leq x^3}
\frac{C_F(p^k)\Lambda_x(p^k)\log(xp)}{p^{k(\sigma+\mathrm{i}t)}}\bigg|^{4n}\mathrm{d}\sigma.
\end{split}
\end{equation*}
Note that
\begin{equation*}
\begin{split}
&\bigg|\sum_{p^k\leq x^3}
\frac{C_F(p^k)\Lambda_x(p^k)\log(xp)}{p^{k(\sigma+\mathrm{i}t)}}\bigg|^{4n}\\
&\ll_n\bigg(\sum_{p^k\leq x^3}
\frac{C_F(p^k)\Lambda_x(p^k)\log(xp)}{p^{k(\sigma+\mathrm{i}t)}}+\sum_{p^k\leq x^3}
\frac{\overline{C_F(p^k)}\Lambda_x(p^k)\log(xp)}{p^{k(\sigma-\mathrm{i}t)}}\bigg)^{4n}\\
&\quad+\bigg(\sum_{p^k\leq x^3}
\frac{C_F(p^k)\Lambda_x(p^k)\log(xp)}{p^{k(\sigma+\mathrm{i}t)}}-\sum_{p^k\leq x^3}
\frac{\overline{C_F(p^k)}\Lambda_x(p^k)\log(xp)}{p^{k(\sigma-\mathrm{i}t)}}\bigg)^{4n}.
\end{split}
\end{equation*}
By using the same argument as in proving \eqref{M_F(t) moment spectral}, for $1\leq k\leq 7$, we get
\begin{equation*}
\begin{split}
\frac{1}{\mathcal{H}}\sum_{F}\frac{h_{T,M}(\mu_F)}{\mathcal{N}_F}\bigg|\sum_{p^k\leq x^3}
\frac{C_F(p^k)\Lambda_x(p^k)\log(xp)}{p^{k(\sigma+\mathrm{i}t)}}\bigg|^{4n}
&\ll_{t,n} (\log x)^{8n}.
\end{split}
\end{equation*}
Thus,
\begin{equation*}
\begin{split}
&\frac{1}{\mathcal{H}}\sum_{F}\frac{h_{T,M}(\mu_F)}{\mathcal{N}_F}\bigg(\int_{1/2}^\infty x^{k(1/2-\sigma)}\bigg|\sum_{p^k\leq x^3}
\frac{C_F(p^k)\Lambda_x(p^k)\log(xp)}{p^{k(\sigma+\mathrm{i}t)}}\bigg|\mathrm{d}\sigma\bigg)^{4n}\\
&\ll_{t,n}\, \frac{1}{(k\log x)^{4n-1}}\int_{1/2}^\infty
x^{k(1/2-\sigma)}(\log x)^{8n}\mathrm{d}\sigma\ll_{t,n} (\log x)^{4n}.
\end{split}
\end{equation*}
Inserting this bound to \eqref{after Cauchy} and using Lemma \ref{lemma 6}, we have the third term in \eqref{R}
is of $O_{t,n}(1)$.
This completes the proof of the lemma.
\end{proof}
\noindent\textbf{Proof of Theorem \ref{main-theorem}.}
By the binomial theorem, we have
\begin{equation}\label{after binomia tm}
\begin{split}
S_F(t)^n=\mathcal{M}_F(t)^n+O_n\bigg(\sum_{\ell=1}^n|\mathcal{M}_F(t)|^{n-\ell}|\mathcal{R}_F(t)|^{\ell}\bigg),
\end{split}
\end{equation}
where $\mathcal{M}_F(t)$ and $\mathcal{R}_F(t)$ are defined in \eqref{MFRF}.
For $1\leq\ell\leq n$, we apply the generalized H\"{o}lder's inequality, Corollary \ref{corollary} and
Lemma \ref{lemma 6 spectral} to get
\begin{equation*}
\begin{split}
&\frac{1}{\mathcal{H}}\sum_{F}\frac{h_{T,M}(\mu_F)}{\mathcal{N}_F}|\mathcal{M}_F(t)|^{n-\ell}|\mathcal{R}_F(t)|^{\ell}\\
\ll &\bigg(\frac{1}{\mathcal{H}}\sum_{F}\frac{h_{T,M}(\mu_F)}{\mathcal{N}_F}\bigg)^{\frac{1}{2}}
\bigg(\frac{1}{\mathcal{H}}\sum_{F}\frac{h_{T,M}(\mu_F)}{\mathcal{N}_F} |\mathcal{M}_F(t)|^{2n}\bigg)^{\frac{n-\ell}{2n}}\\
&\times\bigg(\frac{1}{\mathcal{H}}\sum_{F}\frac{h_{T,M}(\mu_F)}{\mathcal{N}_F}|\mathcal{R}_F(t)|^{2n}\bigg)^{\frac{\ell}{2n}}\\
\ll&_{t,n}\,(\log\log T)^{(n-\ell)/2}\ll_{t,n}(\log\log T)^{(n-1)/2}.
\end{split}
\end{equation*}
Here we have used the positivity of $\mathcal{N}_F$ and the fact that $h_{T,M}(\mu_F)\geq 0$ for all $\mathrm{GL}_3$ Hecke--Maass cusp forms $F$.
Then \eqref{main result} follows from \eqref{M_F(t) moment spectral}, \eqref{after binomia tm}
and the above estimate.

For the second assertion of Theorem \ref{main-theorem},
note that the $n$-th moment of $\mu_{TM}$ is
\begin{equation*}
\int_{\mathbb{R}}\xi^n\mathrm{d}\mu_{TM}(\xi)=\bigg(\sum_{F}\frac{h_{T,M}(\mu_F)}{\mathcal{N}_F}
\bigg(\frac{S_F(t)}{\sqrt{\log\log T}}\bigg)^n\bigg)
\bigg/\bigg(\sum_{F}\frac{h_{T,M}(\mu_F)}{\mathcal{N}_F}\bigg).
\end{equation*}
Using \eqref{main result} and Corollary \ref{corollary}, we obtain that for all $n\geq 0$,
\begin{equation*}
\lim_{T\rightarrow \infty}\int_{\mathbb{R}}\xi^n\mathrm{d}\mu_{TM}(\xi)=C_n=
\sqrt{\pi}\int_{\mathbb{R}}\xi^n \exp(-\pi^2\xi^2)\mathrm{d}\xi.
\end{equation*}
Then by the theory of moments in probability theory (see, for example, \cite[Theorem 30.2]{Billingsley}),
we complete the proof of Theorem \ref{main-theorem}.
\section{Conclusion}
\label{sec:conclusion}
In conclusion, we establish an unconditional asymptotic formula
for the moments of $S_{F}(t)$,
where $F$ is a Hecke--Maass cusp form for $\mathrm{SL}_{3}(\mathbb{Z})$
with Langlands parameter in generic position. Our main result, Theorem~\ref{main-theorem},
extends the conditional result of Liu and Liu~\cite{LL},
which assumed the GRH, to an unconditional setting.
This achievement is made possible by applying a novel weighted zero-density estimate for $L(s, F)$
in the spectral aspect, recently established in \cite{SW1}.
Moreover, we demonstrate that the normalized values $S_{F}(t)/\sqrt{\log \log T}$
converge in distribution to a Gaussian random variable as $T \to \infty$.
A central goal of our ongoing work is to extend this unconditional framework
to $\mathrm{SL}_n(\mathbb{Z})$ for all $n\geq3$, thereby providing
an unconditional counterpart to the conditional theorem of Chen \emph{et al.} \cite{CLW}.

\bigskip
\noindent{\bf Acknowledgements}
Q. Sun was partially supported by the National Natural Science Foundation of China
(Grant Nos. 12471005 and 12031008) and
the Natural Science Foundation of Shandong Province (Grant No. ZR2023MA003).
The authors are very grateful to the
referees for their valuable suggestions.

\bibliographystyle{amsplain}

\end{document}